\documentclass[11pt,reqno]{amsart}
\usepackage{amsmath}
\usepackage{amssymb}
\usepackage[bookmarks]{hyperref}
\usepackage{enumerate}

\newtheorem{theorem}{Theorem}[section]
\newtheorem{lemma}[theorem]{Lemma}

\newtheorem{proposition}[theorem]{Proposition}

\theoremstyle{definition}
\newtheorem{remark}[theorem]{Remark}

\numberwithin{equation}{section}

\allowdisplaybreaks[1]

\begin{document}

\title[Composition operators]{Boundedness and compactness of composition operators on Segal-Bargmann spaces}
\author{Trieu Le}
\address{Department of Mathematics and Statistics, Mail Stop 942, University of Toledo, Toledo, OH 43606}
\email{trieu.le2@utoledo.edu}

\subjclass[2010]{Primary 47B38; Secondary 47B15, 47B33}

\begin{abstract} For $E$ a Hilbert space, let $\mathcal{H}(E)$ denote the Segal-Bargmann space (also known as the Fock space) over $E$, which is a reproducing kernel Hilbert space with kernel $K(x,y)=\exp(\langle x,y\rangle)$ for $x,y$ in $E$. If $\varphi$ is a mapping on $E$, the composition operator $C_{\varphi}$ is defined by $C_{\varphi}h = h\circ\varphi$ for $h\in \mathcal{H}(E)$ for which $h\circ\varphi$ also belongs to $\mathcal{H}(E)$. We determine necessary and sufficient conditions for the boundedness and compactness of $C_{\varphi}$. Our results generalize results obtained earlier by Carswell, MacCluer and Schuster for finite dimensional spaces $E$.
\end{abstract}
\maketitle

\section{Introduction}\label{S:intro}
Let $\mathcal{H}$ be a Banach space of functions on a set $X$ and $\varphi: X\rightarrow X$ be a map. We define the composition operator $C_{\varphi}$ by $C_{\varphi}h = h\circ\varphi$ for all functions $h\in\mathcal{H}$ for which the function $h\circ\varphi$ also belongs to $\mathcal{H}$. We are often interested in the problem of classifying the functions $\varphi$ which induce bounded or compact operators $C_{\varphi}$. There is a vast literature on this problem when $\mathcal{H}$ is the Hardy, Bergman or Bloch space over the unit disc on the plane or the unit ball in $\mathbb{C}^n$ (see, for example, \cite{ColonnaJMAA2009,CowenCRCP1995,Jafari1998,TjaniBAMS2000,Shapiro1993,StevicAMC2010} and references therein). In \cite{CarswellSzeged2003}, Carswell, MacCluer and Schuster studied composition operators on the Segal-Bargmann space (also known as the Fock space) over $\mathbb{C}^n$. They obtained necessary and sufficient conditions on the functions $\varphi$ that give rise to bounded or compact $C_{\varphi}$. They showed any such function $\varphi$ must be affine with an additional restriction. They also provided a formula for the norm of $C_{\varphi}$. This is an interesting result since the problem of computing the norm of a bounded composition operator on other classical function spaces is still an open problem.

Let $n\geq 1$ be an integer. We denote by $d\mu(z)=\pi^{-n}\exp(-|z|^2)dV(z)$ the Gaussian measure on $\mathbb{C}^n$, where $dV$ is the usual Lebesgue volume measure on $\mathbb{C}^n\equiv\mathbb{R}^{2n}$. The Segal-Bargmann (Fock) space $\mathcal{F}_n$ is the space of all entire functions on $\mathbb{C}^n$ that are square integrable with respect to $d\mu$. For $f,g\in\mathcal{F}_n$, the inner product $\langle f,g\rangle$ is given by
\begin{equation*}
\langle f,g\rangle = \int_{\mathbb{C}^n}f(z)\overline{g(z)}\ d\mu(z)= \frac{1}{\pi^n}\int_{\mathbb{C}^n}f(z)\overline{g(z)}\exp(-|z|^2)\ dV(z).
\end{equation*}
It is well known that $\mathcal{F}_n$ contains an orthonormal basis consisting of monomials. In fact, for any multi-index $\alpha=(\alpha_1,\ldots,\alpha_n)$ of non-negative integers, if we put $f_{\alpha}(z) = (\alpha!)^{-1/2}z^{\alpha}$, where $\alpha!=\alpha_1!\cdots \alpha_n!$ and $z^{\alpha} = z_1^{\alpha_1}\cdots z_n^{\alpha_n}$, then $\{f_{\alpha}: \alpha\in\mathbb{Z}_{\geq 0}^{n}\}$ is an orthonormal basis for $\mathcal{F}_n$. It is also well known that $\mathcal{F}_n$ is a reproducing kernel Hilbert space of functions on $\mathbb{C}^n$ with kernel $K(z,w)=\exp(\langle z,w\rangle)$. For more details on $\mathcal{F}_n$, see, for example, Section 1.6 in \cite{Folland1989}. We would like to alert the reader that other authors use slightly different versions of the Gaussian measure (for example, $d\mu(z)=(2\pi)^{-n}\exp(-|z|^2/2)dV(z)$) and hence the resulting reproducing kernels have different formulas (for example, $K(z,w)=\exp(\langle z,w\rangle/2)$). Our choice of the constant here is just for the simplicity of the formulas.

The following theorem \cite[Theorem 1]{CarswellSzeged2003} characterizes bounded and compact composition operators on $\mathcal{F}_n$.

\begin{theorem}[Carswell, MacCluer and Schuster]\label{T:CMS} Suppose $\varphi:\mathbb{C}^n\to\mathbb{C}^n$ is a holomorphic mapping.
\begin{itemize}
\item[(a)] $C_{\varphi}$ is bounded on $\mathcal{F}_n$ if and only if $\varphi(z)=Az+b$, where $A$ is an $n\times n$ matrix with $\|A\|\leq 1$ and $b$ is an $n\times 1$ vector such that $\langle A\zeta,b\rangle=0$ whenever $|A\zeta|=|\zeta|$.
\item[(b)] $C_{\varphi}$ is compact on $\mathcal{F}_n$ if and only if $\varphi(z)=Az+b$, where $\|A\|<1$ and $b$ is any $n\times 1$ vector.
\end{itemize}
\end{theorem}

The norm of $C_{\varphi}$ is given by the next theorem, which is Theorem 4 in \cite{CarswellSzeged2003}. We alert the reader that the  formula presented here is slightly different from the original formula given in \cite{CarswellSzeged2003} because our reproducing kernel is $K(z,w)=\exp(\langle z,w\rangle)$ whereas theirs was $\exp(\frac{1}{2}\langle z,w\rangle)$.

\begin{theorem}\label{T:CMS_norm} Suppose $\varphi(z)=Az+B$, where $\|A\|\leq 1$ and $\langle A\zeta,b\rangle=0$ whenever $|A\zeta|=|\zeta|$. Then the norm of $C_{\varphi}$ on $\mathcal{F}_n$ is given by
\begin{equation}\label{Eqn:CMS_norm}
\|C_{\varphi}\| = \exp\Big(\frac{1}{2}(|w_0|^2-|Aw_0|^2+|b|^2)\Big),
\end{equation}
where $w_0$ is any solution to $(I-A^{*}A)w_0 = A^{*}b$.
\end{theorem}

Motivated by the above results, we study in this paper composition operators on the Segal-Bargmann space $\mathcal{H}(E)$ over an arbitrary Hilbert space $E$. The proof of Theorem \ref{T:CMS}  in \cite{CarswellSzeged2003} makes use of the change of variables and the fact that any $n\times n$ matrix $A$ can be written in the form $A=UDV$ for unitary matrices $U, V$ and a diagonal matrix $D$. Since this approach relies heavily on the finiteness of the dimension of $E$, it does not seem to work when the dimension of $E$ is infinite. It turns out that there is an alternative approach, based on the theory of reproducing kernels. This idea appeared in E. Nordgren's work \cite{Nordgren1978} and it was employed in \cite{JuryAMS2007}, where M. Jury proved the boundedness of composition operators on the Hardy and Bergman spaces of the unit disk without using Littlewood Subordination Principle. We will see that $C_{\varphi}$ is bounded if and only if $\varphi$ is an affine map as in Theorem \ref{T:CMS} but we need a stronger condition on the vector $b$ when $E$ is an infinite dimensional Hilbert space (in the case $E$ has finite dimension, our condition on $b$ is equivalent to the condition in Theorem \ref{T:CMS}). In the course of proving boundedness, we also obtain a formula for $\|C_{\varphi}\|$. Our formula is stated in a different way and it agrees with \eqref{Eqn:CMS_norm} when $E=\mathbb{C}^n$. For the compactness of $C_{\varphi}$, besides the condition that $\varphi(z)=Az+b$ for some linear operator $A$ on $E$ with $\|A\|<1$, it is also necessary that $A$ be a compact operator (this condition is of course superfluous when the dimension of $E$ is finite).

We now state our main results. The first result studies the boundedness and the norm formula for $C_{\varphi}$ on $\mathcal{H}(E)$. 

\begin{theorem}\label{T:bounded_COs} Let $\varphi$ be a mapping from $E$ into itself. Then the composition operator $C_{\varphi}$ is bounded on $\mathcal{H}(E)$ if and only if $\varphi(z)=Az+b$ for $z\in E$, where $A$ is a linear operator on $E$ with $\|A\|\leq 1$ and $A^{*}b$ belongs to the range of $(I-A^{*}A)^{1/2}$.

Furthermore, the norm of $C_{\varphi}$ is given by
\begin{equation}\label{Eqn:norm_formula}
\|C_{\varphi}\| = \exp\Big(\frac{1}{2}\|v\|^2+\frac{1}{2}\|b\|^2\Big),
\end{equation}
where $v$ is the unique vector in $E$ of minimum norm that satisfies the equation $A^{*}b=(I-A^{*}A)^{1/2}v$.
\end{theorem}

The second main result characterizes compact operators $C_{\varphi}$.
\begin{theorem}\label{T:cpt_COs}
Let $\varphi$ be a mapping from $E$ into itself. Then the composition operator $C_{\varphi}$ is compact on $\mathcal{H}(E)$ if and only if there is a compact linear operator $A$ on $E$ with $\|A\| <1$ and a vector $b\in E$ such that  $\varphi(z) = Az+b$ for all $z\in E$.
\end{theorem}

\section{Compositions operators on $\mathcal{H}(E)$}\label{S:COs_Segal-Bargmann}
In the first part of this section we study the space $\mathcal{H}(E)$, where $E$ is an arbitrary Hilbert space. Since Gaussian measure is not available when $E$ is of infinite dimension, our approach here follows the same lines as the construction of the Drury-Arveson space given in \cite{Arveson1998}. In the second part of the section, we consider composition operators on $\mathcal{H}(E)$. Using kernel functions, we provide a criterion for the boundedness of these operators.

\subsection{The construction of $\mathcal{H}(E)$}\label{SS:Segal-Bargmann_space}

For each integer $m\geq 1$, we write $E^{m}$ for the symmetric tensor product of $m$ copies of $E$. We also define $E^{0}$ to be $\mathbb{C}$ with its usual inner product. We have $E^{1}=E$ and for $m\geq 2$, $E^{m}$ is a closed subspace of the full tensor product $E^{\otimes m}$ consisting of all vectors that are invariant under the natural action of the symmetric group $S_m$. The action of $S_m$ on $E^{\otimes m}$ is defined on elementary tensors by
\begin{equation*}
\pi\cdot(x_1\otimes\cdots\otimes x_m) = x_{\pi(1)}\otimes\cdots\otimes x_{\pi(m)}\text{ for } \pi\in S_m \text{ and } x_1,\ldots,x_m\in E.
\end{equation*}

For an element $z\in E$, we write $z^{m}=z\otimes\cdots\otimes z\in E^{m}$ for the $m$-fold tensor product of copies of $z$ (here $z^{0}$ denotes the number $1$ in $E^0=\mathbb{C}$). Each space $E^{m}$ is a Hilbert space with an inner product inherited from the inner product on $E$. We will generally write $\langle\cdot,\cdot\rangle$ for any inner product without referring to the space on which it is defined. The defining space will be clear from the context.

A continuous mapping $p: E\to\mathbb{C}$ is called a continuous $m$-homogeneous polynomial on $E$ if there exists an element $\zeta$ in $E^{m}$ such that $p(z)=\langle z^{m},\zeta\rangle$ for $z\in E$. A continuous mapping $f: E\to\mathbb{C}$ is called a polynomial if $f$ can be written as a finite sum of continuous homogeneous polynomials. In other words, there is an integer $m\geq 0$ and there are vectors $a_0\in\mathbb{C}, a_1\in E^{1},\ldots, a_{m}\in E^{m}$ such that
\begin{equation}\label{Eqn:polynomial}
f(z) = \sum_{j=0}^{m}\langle z^{j},a_j\rangle =  a_0 + \langle z,a_1\rangle + \cdots + \langle z^m,a_{m}\rangle.
\end{equation}

When $E=\mathbb{C}^n$ for some positive integer $n$, the notion of polynomials that we have just given coincides with the usual definition of polynomials in $n$ complex variables. In fact, each polynomial in $z=(z_1,\ldots,z_n)$ is a linear combination of monomials of the form $z_1^{j_1}\cdots z_n^{j_n}$ for non-negative integers $j_1,\ldots, j_n$. Let $\{e_1,\ldots,e_n\}$ denote the standard basis for $\mathbb{C}^n$, where $e_k=(0,\ldots,0,1,0,\ldots)$ with the number $1$ in the $k$th component. Then
\begin{align*}
z_1^{j_1}\cdots z_n^{j_n} & = \langle z,e_1\rangle^{j_1}\cdots\langle z,e_n\rangle^{j_n} \\
& = \langle z^{l}, e_1^{\otimes j_1}\otimes\cdots\otimes e_n^{\otimes j_n}\rangle_{E^{\otimes l}} = \langle z^l,a_l\rangle_{E^{l}},
\end{align*}
where $l=j_1+\cdots+j_n$ and $a_l$ is the orthogonal projection of $e_1^{\otimes j_1}\otimes\cdots\otimes e_n^{\otimes j_n}$ on $E^l$. This shows that any polynomial in the variables $z_1,\ldots, z_n$ can be written in the form \eqref{Eqn:polynomial}.

We denote by $\mathcal{P}_{n}(E)$ the space of all continuous $n$-homogeneous polynomials and $\mathcal{P}(E)$ the space of all continuous polynomials on $E$. For more detailed discussions of polynomials between Banach spaces and locally convex spaces, see \cite{DineenSpringer1999,Mujica1986}.

For two continuous polynomials $f,g$ in $\mathcal{P}(E)$, we can find an integer $m\geq 0$ and vectors $a_j, b_j\in E^{j}$ for $0\leq j\leq m$ such that $f(z)=\sum_{j=0}^{m}\langle z^j,a_j\rangle$ and $g(z)=\sum_{j=0}^{m}\langle z^j,b_j\rangle$. We define
\begin{align}\label{Eqn:inner_product}
\langle f,g\rangle = \sum_{j=0}^{m}j!\ \langle b_j,a_j\rangle.
\end{align}
It can be checked that \eqref{Eqn:inner_product} defines an inner product on $\mathcal{P}(E)$. We denote by $\mathcal{H}(E)$ the completion of $\mathcal{P}(E)$ in the norm induced by this inner product.

There is a natural anti-unitary operator from $\mathcal{H}(E)$ onto the symmetric (boson) Fock space $\mathcal{F}(E) = E^{0}\oplus E^{1}\oplus E^2\oplus \cdots$, where the sum denotes the infinite direct sum of Hilbert spaces. We skip the proof which is straightforward from the definition of $\mathcal{H}(E)$ and $\mathcal{F}(E)$.

\begin{proposition}\label{Prop:Fock_realization}
For each element $f\in\mathcal{P}(E)$ given by formula \eqref{Eqn:polynomial}, we define an element in $\mathcal{F}(E)$ by
\begin{equation*}
Jf = (a_0, \sqrt{1!}\ a_1, \sqrt{2!}\ a_2, \sqrt{3!}\ a_3,\ldots),
\end{equation*}
where $a_j=0$ for $j>m$. Then $J$ is an anti-unitary from $\mathcal{P}_m(E)$ onto $E^{m}$ for each $m\geq 0$ and it extends uniquely to an anti-unitary operator from $\mathcal{H}(E)$ onto $\mathcal{F}(E)$. 
\end{proposition}

As in the case of the Drury-Arveson space, we can realize the elements of $\mathcal{H}(E)$ in more concrete terms, as entire functions on $E$.

\begin{proposition}\label{Prop:analytic_realization}
Each element $f$ in $\mathcal{H}(E)$ can be identified as an entire function on $E$ having a power expansion of the form
\begin{equation*}
f(z) = \sum_{j=0}^{\infty}\langle z^j,a_j\rangle \text{ for all } z\in E,
\end{equation*}
where $a_0\in\mathbb{C}$, $a_1\in E$, $a_2\in E^2,\ldots$. Furthermore, $\|f\|^2 = \sum_{j=0}^{\infty}j!\|a_j\|^2$.

Conversely, if $\sum_{j=0}^{\infty}j!\|a_j\|^2<\infty$, then the power series $\sum_{j=0}^{\infty}\langle z^j,a_j\rangle$ defines an element in $\mathcal{H}(E)$.
\end{proposition}

\begin{proof} By Proposition \ref{Prop:Fock_realization}, each element $f$ has a formal power series of the form
\begin{equation}\label{Eqn:power_expansion}
f(z) = \sum_{j=0}^{\infty}\langle z^m,a_m\rangle,
\end{equation}
where $a_j$ belongs to $E^j$ for $j\geq 0$ and $\sum_{j=0}^{\infty}j!\ \|a_j\|^2 = \|f\|^2<\infty.$ For any $z\in E$, since $\|z^m\| = \|z\|^m$, we have
\begin{align*}
\sum_{j=0}^{\infty}|\langle z^j,a_j\rangle| & \leq \sum_{j=0}^{\infty}\|z^j\|\|a_j\| = \sum_{j=0}^{\infty}\|z\|^j\|a_j\|= \sum_{j=0}^{\infty} \frac{\|z\|^j}{\sqrt{j!}}\ \sqrt{j!}\ \|a_j\|\\
& \leq \Big(\sum_{j=0}^{\infty}\frac{\|z\|^{2j}}{j!}\Big)^{1/2}\Big(\sum_{j=0}^{\infty} j!\ \|a_j\|^2\Big)^{1/2} = \exp(\|z\|^2/2)\|f\|.
\end{align*}
This shows that the power series \eqref{Eqn:power_expansion} converges uniformly on any bounded ball in $E$. It follows that $f$ can be considered as an entire function on $E$.

The converse follows from the fact that the sequence of polynomials $\{p_{m}\}_{m=1}^{\infty}$ defined by $p_m(z)=\sum_{j=0}^{m}\langle z^j,a_j\rangle$ for $m=1,2,\ldots$, is a Cauchy sequence in $\mathcal{H}(E)$.
\end{proof}

For $w$ in $E$, put $$K_w(z)=\exp(\langle z,w\rangle) = \sum_{j=0}^{\infty} \frac{1}{j!}\langle z,w\rangle^j = \sum_{j=0}^{\infty}\Big\langle z^n,\frac{w^j}{j!}\Big\rangle\text{ for } z\in E.$$
By Proposition \ref{Prop:analytic_realization}, $K_w$ belongs to $\mathcal{H}(E)$. For any $f$ given by \eqref{Eqn:power_expansion}, we have
\begin{equation*}
\langle f, K_w\rangle = \sum_{j=0}^{\infty} j!\Big\langle \frac{w^j}{j!},a_j\Big\rangle = f(w)\text{ for } w\in E.
\end{equation*}
Therefore, the function $K(z,w)=K_w(z)$ for $z,w\in E$ is the reproducing kernel function for $\mathcal{H}(E)$. Furthermore, the linear span of the set $\{K_w: w\in E\}$ is dense in $\mathcal{H}(E)$. This shows that $\mathcal{H}(E)$ is a reproducing kernel Hilbert space. For a general theory of these spaces, see, for example, \cite{AronszajnTAMS1950} or \cite[Chapter 2]{AglerAMS2002}.

\begin{remark}
The space $\mathcal{H}(E)$ can be defined in an abstract way by the kernel function $K(z,w)$. However it is not clear from the abstract definition why $\mathcal{H}(E)$ consists of the power series given in Proposition \ref{Prop:analytic_realization}. We have chosen a more concrete construction to exhibit the decomposition
\begin{equation}\label{Eqn:SB_decomposition}
   \mathcal{H}(E)=\bigoplus_{m\geq 0}\mathcal{P}_m(E) = \mathbb{C}\oplus\mathcal{P}_1(E)\oplus\mathcal{P}_2(E)\oplus\cdots,
\end{equation}
which will be useful for us later.
\end{remark}

When $E=\mathbb{C}^n$ for some positive integer $n$, the space $\mathcal{H}(\mathbb{C}^n)$ coincides with $\mathcal{F}_n$, which we discuss in the Introduction.

The following facts are well known in the case $E=\mathbb{C}^n$ and they continue to hold for arbitrary Hilbert space $E$. We skip the proofs, which make use of the fact that a sequence in a reproducing kernel Hilbert space is weakly convergent if and only if it is bounded in norm and it converges point-wise.
\begin{lemma}\label{L:weak_convergence}
The following statements hold in $\mathcal{H}(E)$.
\begin{itemize}
\item[(a)] $\displaystyle\lim_{\|z\|\to\infty}\|K_z\|^{-1}K_z=0$ weakly in $\mathcal{H}(E)$.
\item[(b)] Let $\{u_{m}\}$ be a sequence converging weakly to $0$ in $E$  (in particular, $\{u_m\}$ is bounded). For each $m$, put $f_{m}(z)=\langle z,u_{m}\rangle$ for $z\in E$. Then $\lim_{m\to\infty}f_{m} = 0$ weakly in $\mathcal{H}(E)$.
\end{itemize}
\end{lemma}

\subsection{Composition operators}

For any mapping $\varphi$ from $E$ into itself, we recall that the composition operator $C_{\varphi}$ is defined by $C_{\varphi}h = h\circ\varphi$ for all $h$ in $\mathcal{H}(E)$ for which $h\circ\varphi$ also belongs to $\mathcal{H}(E)$. Since $C_{\varphi}$ is a closed operator, it follows from the closed graph theorem that $C_{\varphi}$ is bounded if and only if $h\circ\varphi$ belongs to $\mathcal{H}(E)$ for all $h\in\mathcal{H}(E)$. 

Now suppose that $C_{\varphi}$ is a bounded operator on $\mathcal{H}(E)$. A priori we do not impose any condition $\varphi$ but it follows from the boundedness of $C_{\varphi}$ that $\varphi$ must be an entire function (at least in the weak sense). In fact, for any $a\in E$, the function $\langle \varphi(\cdot),a\rangle = C_{\varphi}(\langle\cdot,a\rangle)$ belongs to $\mathcal{H}(E)$, hence it is entire on $E$ by Proposition \ref{Prop:analytic_realization}. 

For $z\in E$ and $h\in\mathcal{H}(E)$, since $\langle h,C_{\varphi}^{*}K_z\rangle=\langle C_{\varphi} h,K_z\rangle = h(\varphi(z)) = \langle h,K_{\varphi(z)}\rangle,$ we obtain the well known formula
\begin{equation}\label{Eqn:adjointFormula}
C_{\varphi}^{*} K_z = K_{\varphi(z)}.
\end{equation}
This formula was used in \cite{CarswellSzeged2003} for the proof of the necessity of Theorem \ref{T:CMS}. It turns out that the formula plays an important role in our proof of both the necessity and sufficiency on the boundedness of $C_{\varphi}$.

Let $\mathcal{M}$ denote the linear span of the kernel functions $\{K_{z}: z\in E\}$. We already know that $\mathcal{M}$ is dense in $\mathcal{H}(E)$. Motivated by \eqref{Eqn:adjointFormula}, for any mapping $\varphi$ (even when $C_{\varphi}$ is not a bounded operator on $\mathcal{H}(E)$), we define a linear operator $S_{\varphi}$ with domain $\mathcal{M}$ by the formula
\begin{equation*}
S_{\varphi}\big(\sum_{j=1}^{m}c_jK_{x_j}\big) = \sum_{j=1}^{m}c_jK_{\varphi(x_j)},
\end{equation*}
for distinct elements $x_1,\ldots, x_m$ in $E$ and any complex numbers $c_1,\ldots, c_m$. The operator $S_{\varphi}$ is well defined since the kernel functions $K_{x_1},\ldots, K_{x_m}$ are linearly independent. It follows from \eqref{Eqn:adjointFormula} that if $C_{\varphi}$ is bounded on $\mathcal{H}(E)$, then $S_{\varphi} = C_{\varphi}^{*}$ on $\mathcal{M}$ and hence $S_{\varphi}$ extends to a bounded operator on $\mathcal{H}(E)$. On the other hand, if $S_{\varphi}$ extends to a bounded operator on $\mathcal{H}(E)$, then since $$(C_{\varphi}h)(z) = h(\varphi(z)) =\langle h, K_{\varphi(z)}\rangle = \langle h,S_{\varphi}K_z\rangle = (S_{\varphi}^{*}h)(z)$$ for all $h\in\mathcal{H}(E)$ and all $z\in E$, we conclude that $C_{\varphi}=S_{\varphi}^{*}$ and hence $C_{\varphi}$ is also a bounded operator. It turns out, with the help of kernel functions, that it is more convenient for us to work with $S_{\varphi}$ than with $C_{\varphi}$ directly.

For elements $x_1,\ldots, x_m$ in $E$ and complex numbers $c_1,\ldots, c_m$, since
\begin{align*}
\big\|S_{\varphi}(\sum_{j=1}^{m}c_jK_{x_j})\big\|^2 & = \sum_{j,l}\overline{c}_l c_j\langle K_{\varphi(x_j)},K_{\varphi(x_l)}\rangle = \sum_{j,l}\overline{c}_l c_j K(\varphi(x_l),\varphi(x_j)),
\end{align*}
and
\begin{align*}
\big\|\sum_{j=1}^{m} c_jK_{x_j}\big\|^2 & = \sum_{j,l=1}^{m}\overline{c}_l c_j K(x_l,x_j),
\end{align*}
we see that $S_{\varphi}$ is bounded with $\|S_{\varphi}\|\leq M$ if and only if
\begin{equation}\label{Eqn:semi-definiteness}
\sum_{j,l=1}^{m}c_j\overline{c}_l\big(M^2 K(x_l,x_j) - K(\varphi(x_l),\varphi(x_j))\big)\geq 0.
\end{equation}
Put $\Phi_{M}(z,w) = M^2 K(z,w)- K(\varphi(z),\varphi(w))$ for $z,w\in E$. Since \eqref{Eqn:semi-definiteness} holds for arbitrary $x_1,\ldots, x_m$ in $E$ and arbitrary complex numbers $c_1,\ldots, c_m$, we say that $\Phi_{M}$ is a positive semi-definite kernel on $E$ (in Section \ref{S:boundedCOs} we will discuss more about these kernels). Therefore, $S_{\varphi}$ (and hence, $C_{\varphi}$) is bounded with norm at most $M$ if and only if $\Phi_{M}$ is a positive semi-definite kernel. This criterion for boundedness of composition operators on general reproducing kernel Hilbert spaces was obtained by Nordgren in \cite[Theorem 2]{Nordgren1978}.

Using the formula $K(z,w)=\exp(\langle z,w\rangle)$, we have

\begin{lemma}\label{Lemma:boundedCOs}
For any mapping $\varphi$ from $E$ into itself, the composition operator $C_{\varphi}$ is bounded on $\mathcal{H}(E)$ with norm at most $M$ if and only if the function $$\Phi_{M}(z,w)= M^2\exp(\langle z,w\rangle) - \exp(\langle\varphi(z),\varphi(w)\rangle)$$ is positive semi-definite.

In particular, we have $\Phi_{M}(z,z)\geq 0$, which is equivalent to 
\begin{align*}
M^2\exp(\|z\|^2) \geq \exp(\|\varphi(z)\|^2) \Longleftrightarrow 2\ln M \geq\|\varphi(z)\|^2 - \|z\|^2,
\end{align*}
for all $z\in E$.
\end{lemma}

In Section \ref{S:boundedCOs}, we discuss in more detail positive semi-definite kernels and find conditions on $\varphi$ under which the function $\Phi_{M}$ above is positive semi-definite. Using these conditions we obtain a proof of Theorem \ref{T:bounded_COs}.

\section{Boundedness of composition operators on $\mathcal{H}(E)$}
\label{S:boundedCOs}

\subsection{Positive semi-definite kernels}
Let $X$ be a set. A function $F: X\times X\to\mathbb{C}$ is a positive semi-definite kernel if for any finite set $\{x_1,\ldots, x_m\}$ of points in $X$, the matrix $(F(x_l,x_j))_{1\leq l,j\leq m}$ is positive semi-definite. That is, for any complex numbers $c_1,\ldots, c_m$, we have
\begin{equation*}
\sum_{j,l=1}^{m}\overline{c}_l{c}_j F(x_l,x_j) \geq 0.
\end{equation*}
We list here a few immediate facts about positive semi-definite kernels.
\begin{itemize}
\item[(F1)] Sums of positive semi-definite kernels are positive semi-definite. (A sum here may be an infinite sum provided that it converges point-wise.)
\item[(F2)] Since the Schur (entry-wise) product of two positive semi-definite square matrices is also positive semi-definite, the product of two positive semi-definite kernels is positive semi-definite.
\item[(F3)] Suppose $F$ is a positive semi-definite kernel, then it follows from (F1) and (F2) that the function $\tilde{F}=\exp(F)-1$ is also a positive semi-definite kernel.
\item[(F4)] If there is a vector space $\mathcal{H}$ over $\mathbb{C}$ with inner product $\langle\cdot,\cdot\rangle_{\mathcal{H}}$ and norm $\|\cdot\|_{\mathcal{H}}$ and there is a vector-valued function $f: X\to\mathcal{H}$ such that $F(x,y)=\langle f(y), f(x)\rangle_{\mathcal{H}}$ for $x,y$ in $X$, then for any $x_1,\ldots, x_m$ in $X$ and real numbers $c_1,\ldots, c_m$, we have
\begin{align*}
\sum_{j,l=1}^{m}\overline{c}_l c_j F(x_l,x_j) & = \sum_{j,l=1}^{m}\langle c_jf(x_j),c_lf(x_l)\rangle_{\mathcal{H}} = \|\sum_{j=1}^{m}c_jf(x_j)\|^2_{\mathcal{H}} \geq 0.
\end{align*}
Therefore, $F$ is a positive semi-definite kernel on $X$. It turns out \cite[Theorem 2.53]{AglerAMS2002} that any positive semi-definite kernel can be represented in this form.
\end{itemize}

Now let $E$ be a Hilbert space and $T$ be a bounded linear operator on $E$. Define $F(z,w)=\langle Tz,w\rangle$ for $z,w\in E$. It is clear that if $F$ is positive semi-definite on $E$, then $F(z,z)\geq 0$ for all $z\in E$, which implies that $T$ is a positive operator. Conversely, if $T$ is positive, then since $F(z,w)=\langle T^{1/2}z,T^{1/2}w\rangle$ (here $T^{1/2}$ denotes the positive square root of $T$), it follows from fact (F4) above that $F$ is positive semi-definite. The following proposition provides a generalization of this observation.

\begin{proposition}\label{Prop:quadratic_forms}
Let $u$ be a vector in $E$, $T$ be a self-adjoint operator on $E$, and $M$ be a real number. Define the function
\begin{equation}
F(z,w) = \langle Tz,w\rangle - \langle z,u\rangle - \langle u,w\rangle + M^2 \text{ for } z,w\in E.
\end{equation}
Then the followings are equivalent
\begin{itemize}
   \item[(a)] The function $F$ is a positive semi-definite kernel on $E$.
   \item[(b)] $F(z,z)\geq 0$ for all $z\in E$.
   \item[(c)] $T$ is a positive operator on $E$ and $u=T^{1/2}\hat{u}$ for some $\hat{u}\in E$ with $\|\hat{u}\|\leq M$.
\end{itemize}
Furthermore, if the conditions in (c) are satisfied and $v$ is the vector of smallest norm such that $u=T^{1/2}v$, then we have
\begin{equation}\label{Eqn:sup_quadratic}
\inf\big\{F(z,z): z\in E\big\} = -\|v\|^2+M^2.
\end{equation}
The vector $v$ is characterized by two conditions: (i) $T^{1/2}v=u$ and (ii) $v$ belongs to $\overline{{\rm ran}}(T^{1/2})$.
\end{proposition}

\begin{proof}
It is obvious from the definition of positive semi-definite kernels that (a) implies (b). Now suppose (b) holds. Let $z$ be in $E$. Choose a complex number $\gamma$ such that $|\gamma|=1$ and $\langle\gamma z,u\rangle = \gamma\langle z,u\rangle = |\langle z,u\rangle|$. For any real number $r$, since $F(r\gamma z,r\gamma z)\geq 0$, we obtain
\begin{equation*}
r^2\langle Tz,z\rangle - 2r|\langle z,u\rangle| + M^2 \geq 0.
\end{equation*}
Because this inequality holds for all $r\in\mathbb{R}$ we see that $\langle Tz,z\rangle\geq 0$ and $\big|\langle z,u\rangle\big|^2 \leq M^2\langle Tz,z\rangle$. Since $z$ was arbitrary, we conclude that $T$ is a positive operator and we have $|\langle z,u\rangle|\leq |M|\|T^{1/2}z\|$ for $z\in E$. That this fact implies that $u$ belongs to the range of $T^{1/2}$ is well known but for completeness, we include here a proof. Define a linear functional on the range of $T^{1/2}$ by $\Lambda(T^{1/2}z) = \langle z,u\rangle$, for $z\in E$. By the inequality, $\Lambda$ is well defined and bounded on $T^{1/2}(E)$ with $\|\Lambda\|\leq |M|$. Extending $\Lambda$ to all $E$ by the Hahn-Banach theorem and using the Riesz's representation theorem, we obtain an element $\hat{u}$ in $E$ with $\|\hat{u}\|=\|\Lambda\|\leq |M|$ such that $\Lambda(w)=\langle w,\hat{u}\rangle$ for all $w\in E$. We then have, for any $z\in E$,
\begin{align*}
\langle z,u\rangle & = \Lambda(T^{1/2}z) = \langle T^{1/2}z,\hat{u}\rangle = \langle z,T^{1/2}\hat{u}\rangle.
\end{align*}
Thus $u=T^{1/2}\hat{u}$ and hence (c) follows.

Now assume that (c) holds. For any $z,w$ in $E$,
\begin{align*}
F(z,w) & = \langle T^{1/2}z,T^{1/2}w\rangle - \langle T^{1/2}z,\hat{u}\rangle - \langle \hat{u},T^{1/2}w\rangle + M^2\\
& = \langle T^{1/2}z-\hat{u},T^{1/2}w-\hat{u}\rangle - \|\hat{u}\|^2+M^2.
\end{align*}
Since $-\|\hat{u}\|^2+M^2 \geq 0$, we conclude that $F$ is positive semi-definite. 

Now the preimage of $u$ under $T^{1/2}$ is the non-empty, closed, convex set $\hat{u}+\ker(T^{1/2})$. By a property of Hilbert spaces, there exists a unique vector $v$ of smallest norm in this set. In fact, $v$ is the orthogonal projection of $\hat{u}$ on  $(\ker(T^{1/2}))^{\bot}$. Since $(\ker(T^{1/2}))^{\bot} =\overline{\rm ran}\,(T^{1/2})$, we conclude that $v=P_{\overline{\rm ran}\,(T^{1/2})}\hat{u}$, where $P_{\overline{\rm ran}\,(T^{1/2})}$ is the orthogonal projection from $E$ onto the closure of the range of $T^{1/2}$. Using the facts that $u=T^{1/2}v$ and that $v$ belongs to $\overline{\rm ran}\,(T^{1/2})$, we obtain
\begin{align*}
\inf\big\{F(z,z): z\in E\big\} & = \inf\big\{\|T^{1/2}z-v\|^2: z\in E\big\} - \|v\|^2+M^2 \\
& = -\|v\|^2+M^2.
\end{align*}

To prove the characterization of $v$, let $v'$ be a vector in $\overline{\rm ran}\,(T^{1/2})$ with $u=T^{1/2}v'$. Then the difference $v-v'$ belongs to both $\ker(T^{1/2})$ and $\overline{\rm ran}\,(T^{1/2})$. Since these subspaces are orthogonal complements of each other, we conclude that $v=v'$.
\end{proof}

\begin{remark}\label{R:characterize_v}
In the case $E=\mathbb{C}^n$ for some integer $n\geq 1$, since $\overline{\rm ran}\,(T^{1/2})={\rm ran}\,(T^{1/2})$, the vector $v$ in Proposition \ref{Prop:quadratic_forms} is characterized by $v=T^{1/2}\zeta$ for any $\zeta\in E$ that satisfies $T\zeta = u$.
\end{remark}

\subsection{Bounded composition operators}
We are now ready for the proof of Theorem \ref{T:bounded_COs} on the boundedness of $C_{\varphi}$.

\begin{proof}[Proof of Theorem \ref{T:bounded_COs}]
Suppose first that $C_{\varphi}$ is bounded on $\mathcal{H}(E)$. By Lemma \ref{Lemma:boundedCOs}, for any $z\in E$, $\|C_{\varphi}\|^2\exp(\|z\|^2) - \exp(\|\varphi(z)\|^2)\geq 0$. This implies
\begin{align}\label{Eqn:inequality_COs}
\|\varphi(z)\|^2-\|z\|^2\leq 2\ln\|C_{\varphi}\|.
\end{align}
For a fixed unit vector $a\in E$, put $f_a(z)=\langle z,a\rangle$ and $F_{a}(z)=\langle\varphi(z),a\rangle$. Then $F_{a}$, which equals to $C_{\varphi}(f_a)$, belongs to $\mathcal{H}(E)$. Therefore $F_a$ can be represented as a power series
$$F_a(z) = F_a(0)+\sum_{m=1}^{\infty}\langle z^{m},\zeta_{m}\rangle\text{ for all } z \text{ in } E,$$
where $\zeta_1\in E, \zeta_2\in E^{2}, \ldots$. Now the inequality $|F_a(z)|\leq\|\varphi(z)\|$ together with \eqref{Eqn:inequality_COs} gives $|F_a(z)|^2 - \|z\|^2 \leq 2\ln(\|C_{\varphi}\|)$ for all $z$ in $E$. This implies that $\|\zeta_1\|\leq 1$ and $\zeta_m=0$ for all $m\geq 2$. In particular, $z\mapsto F_a(z)-F_a(0)$ is linear functional with norm at most $1$.

Since the map $z\mapsto\langle\varphi(z)-\varphi(0),a\rangle = F_a(z)-F_a(0)$ is a linear functional with norm at most one for any unit vector $a\in E$, we conclude that $z\mapsto \varphi(z)-\varphi(0)$ is a linear operator with norm at most $1$. Therefore, $\varphi(z) = Az+b$ for some linear operator $A$ on $E$ with $\|A\|\leq 1$ and some vector $b$ in $E$.

Now \eqref{Eqn:inequality_COs} gives $\|z\|^2-\|Az+b\|^2+2\ln(\|C_{\varphi}\|)\geq 0$ for all $z\in E$, which is equivalent to
\begin{align}\label{Eqn:quadratic_COs}
\langle (I-A^{*}A)z,z\rangle - \langle z,A^{*}b\rangle - \langle A^{*}b,z\rangle -\|b\|^2 + 2\ln(\|C_{\varphi}\|)\geq 0.
\end{align}
By Proposition \ref{Prop:quadratic_forms}, we conclude that $A^{*}b$ belongs to the range of $(I-A^{*}A)^{1/2}$. Choose $v\in E$ of smallest norm such that $A^{*}b = (I-A^{*}A)^{1/2}(v)$. Then by Proposition \ref{Prop:quadratic_forms} again, the quantity $$2\ln(\|C_{\varphi}\|)-\|v\|^2-\|b\|^2,$$ being the infimum of the left hand side of \eqref{Eqn:quadratic_COs}, is non-negative. Thus we have
\begin{align}\label{Eqn:norm}
\|C_{\varphi}\| \geq \exp\Big(\frac{1}{2}\|v\|^2+\frac{1}{2}\|b\|^2\Big).
\end{align}

Conversely, suppose $\varphi(z)=Az+b$ such that $\|A\|\leq 1$; $A^{*}b$ belongs to the range of $(I-A^{*}A)^{1/2}$; and $v\in E$ is of smallest norm satisfying $A^{*}b=(I-A^{*}A)^{1/2}(v)$. We will show that the operator $C_{\varphi}$ is bounded on $\mathcal{H}(E)$ with norm at most the quantity on the right hand side of \eqref{Eqn:norm} (hence the inequality in \eqref{Eqn:norm} is in fact an equality). 

We define for $z,w\in E$,
\begin{align*}
F(z,w) & = \langle z,w\rangle - \langle\varphi(z),\varphi(w)\rangle +\|b\|^2+\|v\|^2\\
& = \langle (I-A^{*}A)z,w\rangle - \langle z,A^{*}b\rangle - \langle A^{*}b,w\rangle + \|v\|^2.
\end{align*}
By Proposition \ref{Prop:quadratic_forms}, $F$ is a positive semi-definite kernel, which implies that $\exp(F)-1$ is positive semi-definite. Now let $G$ denote the positive semi-definite kernel defined by $G(z,w)=\exp(\langle\varphi(z),\varphi(w)\rangle)$ for $z,w\in E$. Then $G\cdot(\exp(F)-1)$ is also a positive semi-definite kernel. Since for $z,w\in E$, 
\begin{align*}
G(z,w)\big(\!\exp(F(z,w))-1\!\big) & = \exp(\|b\|^2\!\!+\!\!\|v\|^2)\!\exp(\langle z,w\rangle) - \exp(\langle\varphi(z),\varphi(w)\rangle),
\end{align*}
we conclude, using Lemma \ref{Lemma:boundedCOs}, that $C_{\varphi}$ is bounded on $\mathcal{H}(E)$ and
\begin{align}\label{Eqn:norm_reverse}
\|C_{\varphi}\|\leq\exp\Big(\frac{1}{2}\|b\|^2+\frac{1}{2}\|v\|^2\Big).
\end{align}
This completes the proof of the theorem.
\end{proof}

\begin{remark} 
If $\|A\|<1$, then the operator $I-A^{*}A$ is invertible, hence $(I-A^{*}A)^{1/2}$ is also invertible and as a result, $A^{*}b$ belongs to $(I-A^{*}A)^{1/2}(E)$ for any $b$ in $E$. Theorem \ref{T:bounded_COs} then shows that $C_{\varphi}$ is bounded for any $\varphi$ of the form $\varphi(z)=Az+b$. It turns out (by Theorem \ref{T:cpt_COs}) that $C_{\varphi}$ is in fact compact.
\end{remark}

\subsection{The finite-dimensional case}
We discuss here the case $E=\mathbb{C}^n$ for some positive integer $n$. Suppose $A$ is a bounded operator on $\mathbb{C}^n$ with $\|A\|\leq 1$ and $b$ is a vector in $\mathbb{C}^n$. We claim that $A^{*}b$ belongs to the range of $(I-A^{*}A)^{1/2}$ if and only if $\langle b,A\zeta\rangle=0$ whenever $\|A\zeta\|=\|\zeta\|$. In fact, for $\zeta\in\mathbb{C}^n$, we have 
\begin{align*}
\|\zeta\|^2 - \|A\zeta\|^2 & = \langle\zeta,\zeta\rangle - \langle A^{*}A\zeta,\zeta\rangle 
= \langle (I-A^{*}A)\zeta,\zeta\rangle = \|(I-A^{*}A)^{1/2}\zeta\|^2.
\end{align*}
Therefore $\|A\zeta\|=\|\zeta\|$ if and only if $\zeta$ belongs to $\ker(I-A^{*}A)^{1/2}$. This shows that $\langle b,A\zeta\rangle=0$ for all such $\zeta$ if and only if $A^{*}b$ is in the orthogonal complement of $\ker(I-A^{*}A)^{1/2}$, which is $\overline{{\rm ran}}\,(I-A^{*}A)^{1/2}$. On $\mathbb{C}^n$, the identity $\overline{{\rm ran}}\,(I-A^{*}A)^{1/2} = {\rm ran}\,(I-A^{*}A)^{1/2}$ holds, so the claim follows. We then recover Theorem \ref{T:CMS}.

Also, by Remark \ref{R:characterize_v}, the vector $v$ in \eqref{Eqn:norm_formula} is characterized by $v=(I-A^{*}A)^{1/2}w_0$ for any $w_0\in \mathbb{C}^n$ satisfying $(I-A^{*}A)w_0 = A^{*}b$. It then follows that $$\|v\|^2+\|b\|^2 = \|(I-A^{*}A)^{1/2}w_0\|^2 +\|b\|^2 = \|w_0\|^2 - \|Aw_0\|^2 + \|b\|^2.$$ Therefore, we recover the norm formula given in Theorem \ref{T:CMS_norm}.

\subsection{The infinite-dimensional case}
As we have seen above, the requirement that $\langle A\zeta,b\rangle=0$ whenever $\|A\zeta\|=\|\zeta\|$ is equivalent to the requirement that $A^{*}b$ belongs to the \emph{closure} of the range of $(I-A^{*}A)^{1/2}$. In the case $E$ has infinite dimension, this certainly does not imply that $A^{*}b$ belongs to the range of $(I-A^{*}A)^{1/2}$ and hence, by Theorem \ref{T:bounded_COs}, the composition operator $C_{\varphi}$ (with $\varphi(z)=Az+b)$ may not be bounded on $\mathcal{H}(E)$. 

We provide here a concrete example. Let $E$ be a separable Hilbert space with an orthonormal basis $\{v_{m}: m=1, 2, \ldots\}$. Let $A$ be the diagonal operator with $Av_m = \alpha_m v_m$ where $(1-m^{-3})^{1/2}<\alpha_m<1$ for all integers $m\geq 1$. Put $b=\sum_{m=1}^{\infty}m^{-1}v_m$, which belongs to $E$. Since $\|A\zeta\|=\|\zeta\|$ if and only if $\zeta=0$, we see that the condition $\langle A\zeta,b\rangle=0$ whenever $\|A\zeta\|=\|\zeta\|$ holds trivially.

Define $\varphi(z)=Az+b$ for $z\in E$. We claim that 
\begin{equation}\label{Eqn:unboundedCOs}
\sup_{z\in E}\big(\|\varphi(z)\|-\|z\|^2\big)=\infty
\end{equation} and hence, by Lemma \ref{Lemma:boundedCOs}, the operator $C_{\varphi}$ is not bounded on $\mathcal{H}(E)$. For any integer $m\geq 1$, put $t_m = \alpha_m (1-|\alpha_m|^2)^{-1}m^{-1}$. A simple calculation gives
\begin{align*}
\|\varphi(t_mv_m)\|^2 - \|t_mv_m\|^2 & = \|t_mAv_m + b\|^2-t_m^2 \geq (t_m\alpha_m+m^{-1})^2 - t_m^2\\
& = -(1-|\alpha_m|^2)t_m^2 + 2t_m m^{-1}\alpha_{m} +m^{-2}\\
& = \alpha_m^2 (1-\alpha_m^2)^{-1}m^{-2} + m^{-2}\\
& = (1-\alpha_m^2)^{-1}m^{-2}\\
& > m \quad(\text{since } (1-\alpha_m^2)<m^{-3}).
\end{align*}
This then gives \eqref{Eqn:unboundedCOs}.

\section{Compactness of composition operators on $\mathcal{H}(E)$}
\label{S:compactCOs}
In this section we characterize mappings $\varphi$ that induce compact operators $C_{\varphi}$ on $\mathcal{H}(E)$. Before discussing the general case, let us consider first the case $\varphi(z)=Az$, where $A$ is a linear operator on $E$ with $\|A\|\leq 1$. In what follows, we will simply write $C_{A}$ for $C_{\varphi}$.

It turns out that via the anti-unitary $J$ that we have seen in Proposition \ref{Prop:Fock_realization}, the operator $C_{A}$ has an easy description. Let $f$ be a continuous $m$-homogeneous polynomial on $E$. Then there is an element $a_m\in E^{m}$ such that $f(z)=\langle z^m,a_m\rangle$ for $z\in E$. This gives
\begin{align*}
(C_{A}f)(z) = \langle (Az)^m, a_m\rangle = \langle A^{\otimes m}(z^m),a_m\rangle = \langle z^{m}, (A^{*})^{\otimes m}a_m\rangle,
\end{align*}
where $A^{\otimes m}$ denotes the tensor product of $m$ copies of $A$. We conclude that $C_{A}f$ is also a continuous $m$-homogeneous polynomial. Therefore, the space $\mathcal{P}_m(E)$ of continuous $m$-homogeneous polynomials is invariant under $C_{A}$ and we have the identity $C_{A}|_{\mathcal{P}_m(E)} = J^{-1}(A^{*})^{\otimes m}J$. This, together with the decomposition \eqref{Eqn:SB_decomposition}, gives
\begin{align}\label{Eqn:direct_sum}
C_{A} = J^{-1}\big(1 \oplus A^{*}\oplus (A^{*})^{\otimes 2} \oplus (A^{*})^{\otimes 3} \oplus\cdots\big)J,
\end{align}
where the sum is an infinite direct sum of operators. The identity \eqref{Eqn:direct_sum} shows that $C_{A}$ is compact if and only if $(A^{*})^{\otimes m}$ is compact for each $m\geq 1$ and $\|(A^{*})^{\otimes m}\|\rightarrow 0$ as $m\to\infty$. Using the fact that $(A^{*})^{\otimes m}$ is compact if and only if $A^{*}$ (and hence $A$) is compact and the well known identity $\|(A^{*})^{\otimes m}\| = \|A^{*}\|^m=\|A\|^m$, we conclude that $C_{A}$ is compact if and only if $A$ is compact and $\|A\|<1$. We have thus proved a special case of Theorem \ref{T:cpt_COs}. A proof of the full version of Theorem \ref{T:cpt_COs} is given below.

\begin{proof}[Proof of Theorem \ref{T:cpt_COs}]
Assume first that $C_{\varphi}$ is compact. By Theorem \ref{T:bounded_COs}, there is a linear operator $A$ on $E$ with $\|A\|\leq 1$ and a vector $b\in E$ such that $A^{*}b\in (I-A^{*}A)^{1/2}(E)$ and $\varphi(z)=Az+b$ for all $z\in E$. We will show that $A$ is compact and $\|A\|<1$.

Let $\{u_m\}_{m=1}^{\infty}$ be a sequence in $E$ that converges weakly to zero. For each $m$, put $f_{m}(z)=\langle z,u_m\rangle$ for $z\in E$. Then $f_m\to 0$ weakly as $m\to\infty$ by Lemma \ref{L:weak_convergence}. This implies that $\lim_{m\to\infty}\|C_{\varphi}f_m\| = 0$. But
\begin{align*}
(C_{\varphi}f_m)(z) & = f_m(\varphi(z))=\langle Az+b,u_m\rangle = \langle z,A^{*}u_m\rangle + \langle b,u_m\rangle,
\end{align*}
so $\|C_{\varphi}f_m\|^2 = \|A^{*}u_m\|^2+|\langle b,u_m\rangle|^2$. We then have $\lim_{m\to\infty}\|A^{*}u_m\|^2=0$. Therefore, $A^{*}$ is a compact operator and hence, $A$ is also compact.

Suppose it were true that $\|A\|=1$. Then $\|A^{*}A\|=1$. Since $A^{*}A$ is a positive compact operator, $1$ is its eigenvalue. So there is a vector $w\neq 0$ such that $A^{*}Aw=w$, which is equivalent to $(I-A^{*}A)^{1/2}w=0$. Since $A^{*}b$ belongs to the range of $(I-A^{*}A)^{1/2}$, we infer that $\langle w,A^{*}b\rangle=0$, or equivalently, $\langle Aw,b\rangle=0$. For any real number $r$, the identity $C^{*}_{\varphi}(K_{rw}) = K_{\varphi(rw)}$ together with a computation reveals
\begin{align*}
\Big\|C_{\varphi}^{*}\Big(\frac{K_{rw}}{\|K_{rw}\|}\Big)\Big\|^2 & = \frac{\|K_{\varphi(rw)}\|^2}{\|K_{rw}\|^2} = \exp\big(\|\varphi(rw)\|^2-\|rw\|^2\big)\\
& = \exp\big(\|rAw+b\|^2-r^2\|w\|^2\big) = \exp(\|b\|^2).
\end{align*}
Since $K_{rw}/\|K_{rw}\|\to 0$ weakly as $r\to\infty$ by Lemma \ref{L:weak_convergence} again, it follows that $C_{\varphi}^{*}$ is not a compact operator. Hence $C_{\varphi}$ is not compact either. This gives a contradiction. Therefore, we have $\|A\|<1$.

Conversely, suppose $\varphi(z)=Az+b$, where $A$ is a compact operator on $E$ with $\|A\|<1$ and $b$ is an arbitrary vector in $E$. Choose a positive number $\alpha$ such that $\|A\|<\alpha<1$. Put $\varphi_1(z)=\alpha^{-1}Az$ and $\varphi_2(z)=\alpha z+b$ for $z\in E$. Then as we have shown above, $C_{\varphi_1}$ is compact. By Theorem \ref{T:bounded_COs}, $C_{\varphi_2}$ is bounded. Since $\varphi = \varphi_2\circ\varphi_1$, it follows that $C_{\varphi} = C_{\varphi_1}C_{\varphi_2}$ and hence $C_{\varphi}$ is a compact operator.
\end{proof}


\section{Normal, isometric and co-isometric composition operators}
As consequences of Theorem \ref{T:bounded_COs}, we determine in this section the mappings $\varphi$ that give rise to normal, isometric or co-isometric operators $C_{\varphi}$ (recall that an operator on the Hilbert space is called co-isometric if its adjoint is an isometric operator). We will make use of the formulas $C_{\varphi}K_0 = K_0$, $C_{\varphi}^{*}K_0 = K_{\varphi(0)}$ (by \eqref{Eqn:adjointFormula}) and $C_{\varphi}^{*}C_{\varphi}K_0=K_{\varphi(0)}$, where $K_0\equiv 1$ is the reproducing kernel function of $\mathcal{H}(E)$ at $0$.

We first show that if $C_{\varphi}$ is either a normal, isometric or co-isometric operator on $\mathcal{H}(E)$, then $\varphi(0)=0$. The argument is fairly standard. In fact, if $C_{\varphi}$ is normal, then we have $\|C_{\varphi}^{*}K_0\| = \|C_{\varphi}K_0\|$, which gives $\|K_{\varphi(0)}\|=\|K_0\|$. If $C_{\varphi}$ is isometric, then $C^{*}_{\varphi}C_{\varphi}K_0 = K_0$, which gives $K_{\varphi(0)}=K_0$ and hence, in particular, $\|K_{\varphi(0)}\|=\|K_0\|$. If $C_{\varphi}$ is co-isometric then we also have $\|K_0\|=\|C^{*}_{\varphi}K_0\| = \|K_{\varphi(0)}\|$. Since $\|K_{\varphi(0)}\|^2=\exp(-\|\varphi(0)\|^2)$ and $\|K_0\|^2=1$, we conclude that in each of the above cases, $\varphi(0)=0$.

Now since $\varphi(0)=0$, Theorem \ref{T:bounded_COs} shows that $\varphi(z)=Az$ for some operator $A$ on $E$ with $\|A\|\leq 1$. Then $C_{\varphi}=C_{A}$ and $C_{\varphi}^{*}=C_{A^{*}}$ and hence
\begin{equation*}
C^{*}_{\varphi}C_{\varphi} = C_{A^{*}}C_{A} = C_{AA^{*}} \quad\text{ and }\quad C_{\varphi}C_{\varphi}^{*} = C_{A}C_{A^{*}} = C_{A^{*}A}.
\end{equation*}
We obtain

\begin{proposition}\label{P:normalCOs} Let $\varphi$ be a mapping on $E$ such that $C_{\varphi}$ is a bounded operator on $\mathcal{H}(E)$.
\begin{itemize}
\item[(a)] $C_{\varphi}$ is normal if and only if there exists a normal operator $A$ on $E$ with $\|A\|\leq 1$ such that $\varphi(z)=Az$ for all $z\in E$.
\item[(b)] $C_{\varphi}$ is isometric if and only if there exists a co-isometric operator $A$ on $E$ such that $\varphi(z)=Az$ for all $z\in E$.
\item[(c)] $C_{\varphi}$ is co-isometric if and only if there exists an isometric operator $A$ on $E$ such that $\varphi(z)=Az$ for all $z\in E$.
\end{itemize}
\end{proposition}

\begin{remark} Statement (a) in Proposition \ref{P:normalCOs} holds also for composition operators on the Hardy and Bergman spaces of the unit ball (see \cite[Theorem 8.1]{CowenCRCP1995}), where an analogous result to Theorem \ref{T:bounded_COs} is not available. (In fact, on the Hardy and Begrman spaces, mappings that are not affine can still give rise to bounded composition operators.) The proof of \cite[Theorem 8.1]{CowenCRCP1995} can be adapted to prove Proposition \ref{P:normalCOs} (a) without appealing to Theorem \ref{T:bounded_COs} in the case $E$ has finite dimension. On the other hand, since that proof relies on the finiteness of the dimension, it does not seem to work when $E$ has infinite dimension.
\end{remark}

\begin{remark} In the case $E=\mathbb{C}^n$ for some positive integer $n$, isometric operators on $E$ are also co-isometric and vice versa, and all these operators are unitary. Statements (b) and (c) in Proposition \ref{P:normalCOs} then imply that $C_{\varphi}$ is isometric on $\mathcal{F}_n$ if and only if it is co-isometric if and only if it is unitary. 
\end{remark}

\bibliography{MyRefs}
\bibliographystyle{amsplain}
\end{document}